\documentclass[12pt]{amsart} % use larger type; default would be 10pt

\usepackage[utf8]{inputenc} % set input encoding (not needed with XeLaTeX)

\usepackage{geometry} % to change the page dimensions
\usepackage{graphicx} % support the \includegraphics command and options

\usepackage{tikz-cd}% for diagrams
\usepackage{amssymb}
\usepackage{array} % for better arrays (eg matrices) in maths
\usepackage{paralist} % very flexible & customisable lists (eg. enumerate/itemize, etc.)
\usepackage{subfig} % make it possible to include more than one captioned figure/table in a single float
\usepackage{amsthm}
\newtheorem{theorem}{Theorem}[section]

\newtheorem{lemma}[theorem]{Lemma}
\newtheorem{corollary}[theorem]{Corollary}
\newtheorem{definition}[theorem]{Definition}

\usepackage{color} % to add colors

\begin{document}

\title[]{Determinant bundle over the universal \\moduli space of principal bundles \\over the Teichm{\"u}ller space}
\begin{abstract}
For a Riemann surface $X$ and the moduli of regularly stable $G$-bundles $M$, there is a naturally occuring ``$adjoint$" vector bundle over $X \times M$. One can take the determinant of this vector bundle with respect to the projection map onto $M$. Our aim here is to study the curvature of the determinant bundle as the conformal structure on $X$ varies over the Teichm\"{u}ller space. 
\end{abstract}
\author[]{Arideep Saha} 

\address{School of Mathematics, Tata Institute of Fundamental 
Research, Homi Bhabha Road, Mumbai 400005, India} 

\email{arideep@math.tifr.res.in}
\date{}



\maketitle

\section{Introduction}

Let $X$ be a Riemann surface of genus $g \geq 3$ with a Riemannian metric. Let $G$ be a semisimple linear algebraic group defined over $\mathbb{C}$. Let $M$ denote the moduli of regularly stable principal $G-$bundles over $X$. Let $\mathcal{T}_g$ denote the Teichm{\"u}ller space, the space of all conformal structures on a genus $g$ smooth curve. Let $\mathcal{C}_g$ denote the universal family of Riemann surfaces of genus $g$ over the Teichm{\"u}ller space $\mathcal{T}_g$. Let $\mathcal{M}$ denote the ``universal moduli space" over $\mathcal{T}_g$, which we would construct here. $\mathcal{M}$ is such that the fibre over $t \in \mathcal{T}_g$ is the moduli of regularly stable principal $G$-bundles over the Riemann surface represented by the point $t$.

\par

Our aim here is to first construct the ``universal adjoint bundle" $ad(\mathcal{P})$ over $\mathcal{C}_g \times_{\mathcal{T}_g} \mathcal{M}$. Let $p_2:\mathcal{C}_g \times_{\mathcal{T}_g} \mathcal{M} \longrightarrow \mathcal{M}$ be the projection map. Denote by $\Theta$ the top exterior product of the first direct image of $ad(\mathcal{P})$, i.e. $ \Theta :=\bigwedge ^{top}(R^1 p_{2\ast} ad(\mathcal{P}))$. By a general construction of Bismut, Gillet and Soul{\'e} \cite{B1,B2,B3} we have a hermitian metric on the line bundle $\Theta$. The above mentioned authors also give a general formula for the curvature of the connection corresponding to the constructed hermitian metric on $\Theta$.

\par

We show here that in our particular situation, the curvature form of $\Theta$ conincides with a natural $(1,1)$ form on $\mathcal{M}$. We prove that the curvature form $K(\Theta)$ has the expression:

\centerline{$K(\Theta)=-2 \pi \iota . q^{\ast} \Omega +(r/{6 \pi} ) \iota . \sigma ^{\ast} \omega$}

where $\Omega$ is a naturally occuring K\"{a}hler form on  $M$ and $\omega$ corresponds to the Weil-Petersson metric  on $\mathcal{T}_g$. The maps $q$ and $\sigma$ are projections onto $M$ and $\mathcal{T}_g$, respectively. The work here generalizes the case of vector bundles, done by Indranil Biswas \cite{B}, to the case of principal $G$-bundles. 

\par

We divide the whole content into four sections. The first section discusses some preliminary results. The second section discusses the construction and the holomorphic structure of the universal moduli space $\mathcal{M}$. The third section discusses the construction of the universal adjoint bundle and the determinant bundle with the connection corresponding to the hermitian structure. Finally, in the fourth section we calculate the curvature of the connection using the formula given by Bismut, Gillet and Soul{\'e}. 

\textbf{Acknowledgement.} I would like to thank Prof. I. Biswas for many helpful discussions. I would also like to thank Prof. N. Fakhruddin for pointing out the specific reference \cite{F}.

\section{Preliminaries}

Before going into the details of the construction of universal moduli space we briefly recall a few definitions and results related to $G$-bundles in the following subsection. 

\subsection{$G$-bundles}

In the following definitions and results of this subsection we shall assume $G$ to be any linear algebraic group over $\mathbb{C}$, i.e. not necessarily semisimple.

\begin{definition}
(principal $G$-bundle) An algebraic principal bundle over a scheme $X$ with structure group $G$  is a scheme $E$ on which there is a right action of $G$ and a $G$-invariant morphism $\pi:E \longrightarrow X$ such that for any point $x \in X$ there is a Zariski open neighbourhood $U$ and a faithfully flat morphism $f:U' \longrightarrow U$ and an $G$-equivariant isomorphism $\phi:f^{\ast}(E) \longrightarrow U' \times G$ over $U'$. Here $G$ operates on $U' \times G$ by right translation on the second factor.   
\end{definition}  

\par

Since in our case $G$ is a linear algebraic group (i.e. affine) and $X$ is a smooth Riemann surface, we note down a simplification of the above definition.

\par

\textbf{Note:-} Since $G$ is a linear algebraic group any analytic $G$-bundle on the compact Riemann surface $X$ has a unique algebraic $G$-bundle structure. Also any analytic morphism between two $G$-bundles over $X$ is an algebraic morphism. Because of this equivalence of categories, henceforth we would freely interchange between algebraic and analytic category of principal $G$-bundles on $X$ and call it a $G$-bundle without specifying whether it belongs to algebraic category or analytic category. 

\par

Next we define the notion of stability (semistability) as introduced by Ramanathan \cite{R1}. Recall that a subgroup $P$ of $G$ is called a $parabolic$ subgroup of $G$ if $G/P$ is a complete variety.

\begin{definition}
A $G$-bundle $E \longrightarrow X$ is called stable (semistable) if for any reduction of structure group to any maximal parabolic subgroup $P$, i.e. $\sigma : X \longrightarrow E/P$, we have deg $\sigma ^{\ast} (T_{E/P}) > 0$ ($\geq 0$), where $T_{E/P}$ is the tangent bundle along the fibres of $E/P \longrightarrow X$. 
\end{definition}

\begin{definition}
A $G$-bundle $E \longrightarrow X$ is called regularly stable if the automorphism group of $E$ is equal to $Z$, the center of $G$.
\end{definition}

Next we recall the definitions of connection and curvature on a $G$-bundle.

\begin{definition}
(Connection) Let $E$ over $X$ be a $G$-bundle, then a $G$-connection on $E$ is a section $\bar{h}: T_{X} \longrightarrow T_{E}/G$.
\end{definition}

\begin{definition}
(Curvature) Let $h$ be a connection on a principal $G$-bundle $E$ over $X$. Then the curvature $K(h)$ of the connection $h$ is defined by $K(h)(X,Y)=[h(X),h(Y)]-h[X,Y]$, where $X,Y \in T_X$ and $[ , ]$ is the lie algebra structure on the respective tangent spaces. 
\end{definition}

\par

We now recall and discuss briefly the results of Ramanathan \cite{R1} which relates moduli of semistable $G$-bundles on a Riemann surface $X$ and certain representations of $\pi _1 (X)$, the fundamental group of $X$. 

\begin{definition}
Let $H \subset K$ be a subgroup of a maximal compact subgroup $K$. $H$ is said to be \textbf{\textit{irreducible}} if \\

\centerline{$\Big \{ g \in G | hgh^{-1} = g, \forall h \in H \Big \} = $center of $G.$}
\end{definition}

\begin{definition}
A representation $\rho: \pi \longrightarrow K$ is said to be irreducible if the image subgroup is irreducible. 
\end{definition}

\textbf{Proposition 7.7} of \cite{R1} coupled with \textbf{Theorem 7.1} of [R1] says that, for $G$ semisimple, there is a 1-1 correspondence between regularly stable $G$-bundles and the irreducible representations of $\pi_1(X)$, modulo equivalences.

\par

\textbf{Note:} The definition of \textit{irreducibility}, given in this article, differs slightly from the definition given in \cite{R1}. Hence, although irreducible representations (as defined in \cite{R1}) correspond to stable $G$-bundles, we would be considering the subset of stable 
$G$-bundles, namely the \textit{regularly stable} $G$-bundles, as they precisely correspond to representations irreducible in our sense. Also, when the genus of $X$ is $ \geq 3$, the smooth locus in the moduli of $G$-bundles coincides with the regularly stable $G$-bundles.

\subsection{Teichm\"{u}ller space}

In this subsection we recall some known facts about the Teichm\"{u}ller spaces.

\par

Let $S$ be a compact connected oriented $C^{\infty}$ surface of genus $g \geq 2$. The space of all complex structures on $S$ compatible with the orientation of $S$, denoted by $Com(S)$, has a natural structure of an infinite dimensional complex Fr\'{e}chet manifold. 

\par

Let $Diff^{+}(S)$ be the group of all orientation preserving diffeomorphism of $S$. $Diff^{+}(S)$ has a natural action on $Com(S)$, given by the pushforward of a complex structure by a diffeomorphism. Let $Diff^{0}(S) \subset Diff^{+}(S)$ be the subgroup consisting of all the diffeomorphisms of $S$ which are homotopic to the identity map.The Teichm\"{u}ller space for $S$, denoted by $\mathcal{T}_g$, is defined as \\

\centerline{$\mathcal{T}_g = Com(S)/Diff^{0}(S)$} 

For $f \in Diff^{+}(S)$ we know that $f$ preserves the complex structure of $Com(S)$. Hence there is a natural induced complex structure on $\mathcal{T}_g$.

\par

To give a more elaborate description of $\mathcal{T}_g$, consider pairs of the form $(X,f)$ where $X$ is a Riemann surface and $f:X \longrightarrow S$ is a diffeomorphism. Identify two pairs $(X,f)$ and $(Y,g)$ if there is a biholomorphic map $h:X \longrightarrow Y$ satifying the commuting diagram

\[
\begin{tikzcd}
X \arrow{d}{f} \arrow{r}{h} & Y \arrow{d}{g}\\ 
S \arrow{r}{h'} & S
\end{tikzcd}
\]

with $h' \in Diff^{0}(S)$. The Teichm\"{u}ller space $\mathcal{T}_g$ is the moduli space of equivalence classes of such pairs.

\par

For any point $t(=(X,f))\in \mathcal{T}_g$, the holomorphic tangent space $T_{\mathcal{T}_g ,t}$ can be identified with $H^{1}(X,T_{X})$. By Serre duality we have $T_{\mathcal{T}_g ,t} ^{\ast} = H^{0}(X,K_{X}^{2})$, where,  $K_X$ is the holomorphic cotangent bundle on $X$. We also know that $X$ has a unique metric with curvature -1, known as the Poincar\'{e} metric, which we denote by $g_X$. Using the Poincar\'{e} metric $g_X$, we give a bilinear pairing $Q$ on $H^0(X,K_{X}^{2})$ defined by \\

\centerline{$Q(\alpha , \beta)=\int _{X} \alpha \otimes \bar{\beta} \otimes g_{X}^{-1}$}

The above pairing $Q$ defines a hermitian metric. In this way we get a Riemannian metric on $\mathcal{T}_g$. It can be shown that this metric, known as the \textit{Weil-Petersson} metric, is a K\"{a}hler metric.

\section{Universal moduli space}

\par

  Henceforth we fix the notations and conventions introduced in the introduction.
  
\par

  Let $K$ be a maximal compact subgroup of $G$. By Ramanathan's result (Proposition 7.7 \cite{R1}) we know, $M$= $Hom_{ir}(\pi_1(X),K)/K$ where $Hom_{ir}$ means the irreducible representations and the action of $K$ on $Hom_{ir}(\pi_1(X),K)$ comes via its conjugation action on $K$. We also know that there is a complex structure on $M$ (infact an algebraic structure) since $M$ is the moduli of $G$-bundles over the algebraic curve $X$.

\par

Our aim in this section is to construct a complex manifold $\mathcal{M}$ with a holomorphic map \\

\centerline{$\phi$:$\mathcal{M} \rightarrow\mathcal{T}_g $}

 such that fibre over t $\in \mathcal{T}_g$ is the moduli of regularly stable $G$-bundles over $X_t$. We denote by $M_{t}$ the complex manifold of moduli of regularly stable $G$-bundles on $X_t$. We denote $Hom_{ir}(\pi_1(X_t),K)/K$ by $R$ which is the underlying smooth manifold for $M_{t}$ for all $t \in \mathcal{T}_g$. Let us denote by $S$ the underlying Riemann surface of $X_t$. Note that $X_t$ is diffeomorphic to $S$ for all $t \in \mathcal{T}_g$. 
  
\par

  To construct $\mathcal{M}$ as a complex manifold we give a complex structure on $\mathcal{T}_g\times R$ such that $\phi$ becomes holomorphic. So first, we shall give an almost complex structure and then show it is integrable. 
  
\par

  Let $(t,m)\in \mathcal{T}_g\times R$. Since $\mathcal{T}_g$ is a complex manifold it has an (integrable) almost complex structure $J_1$, say.  Similarly, since $t\times R$ has the structure of the complex manifold $M_{t}$, $t\times R$  has an (integrable) almost complex structure, $J_t$ say.  The two almost complex structures, $J_1$ and $J_t$ give an almost complex structure on $T_{(t,m)}(\mathcal{T}_g\times R)$ given by $J_1\oplus J_t$. This also shows  $T^{(1,0)}_{(t,m)}(\mathcal{T}_g\times R)= T^{(1,0)}_{(t)}(\mathcal{T}_g)\oplus T^{(1,0)}_{t,m}(R)$. Note that the almost complex structure on $T^{(1,0)}_{t,m}(R)$ is dependent on $t \in \mathcal{T}_g$.

\par

To show integrability we use Newlander-Nirenberg criterion. We show that if $\chi , \psi \in T^{(1,0)}_{(t,m)}(\mathcal{T}_g \times R)$, then $[\chi, \psi] \in T^{(1,0)}_{(t,m)}(\mathcal{T}_g \times R)$. Although the integrability can be shown by a direct calculation, we use construction of moduli spaces by Faltings' (see \cite{F}) to quickly understand the integrability of the almost complex structure on $\mathcal{T}_g \times R$.
  
\par

We first recall the definition of the mapping class group, denoted by $MCG$. For a $C^{\infty}$ Riemann surface $S$ the mapping class group is defined as:\\

\centerline{$MCG=Diff^{+}(S)/Diff^0(S)$.}

where $Diff^{+}(S)$ and $Diff^{0}(S)$ are, as defined at the beginning of the subsection 2.2, orientation preserving diffeomorphisms of $S$ and diffeomorphisms homotopic to the $identity$ morphism of $S$ respectively.

\par

We have an action of $MCG$ on $\mathcal{T}_g$ and $R$ which preserves the respective almost complex structures. The action of $MCG$ on $\mathcal{T}_g$ is the natural action which is induced from the action which pulls back complex structures on $S$ via diffeomorphisms. $MCG$ acts on the representation space $R$ via the action of diffeomorphisms on $\pi_{1}(S)$. Here we note that, although for each $t \in \mathcal{T}_g$, $R$ has a different almost complex structure, the action of $MCG$ preserves all such complex structures. 

\par

Let $U \subset \mathcal{T}_g/MCG$ be the set of curves with no non-trivial (holomorphic) automorphisms. $U$ is zariski open in $\mathcal{T}_g/MCG$. Using the constructions of moduli of $G$-bundles over curves parametrized by varieties as in \cite{F}, we have universal moduli of $G$-bundles, $\bar{\mathcal{M}}$, over curves parametrized by $U$. The algebraic structure of $\bar{\mathcal{M}}$ gives a complex structure, and hence, an integrable almost complex structure on $\bar{\mathcal{M}}$ defined over the complex manifold $U$. 

\par

Let, \\
\centerline{$q: \mathcal{T}_g \longrightarrow \mathcal{T}_g/MCG$}

be the quotient map. Both $U$ and $q^{-1}(U)$ are dense in $\mathcal{T}_g/MCG$ and $\mathcal{T}_g$ respectively. We know that the action of $MCG$ on $q^{-1}(U)$ is discrete and $q^{-1}(U)$ is a covering space over $U$. 

\par

We note that, $q^{\ast}\bar{\mathcal{M}}$, which is the moduli of $G$-bundles over curves parametrized by $q^{-1}(U)$, is canonically diffeomorphic to  $q^{-1}(U) \times R$. The diffeomorphism is given by identifying $R$ with the fibre $q^{\ast}\bar{\mathcal{M}}_u$ over $u \in q^{-1}(U)$ using the results of Ramanathan which identifies moduli of $G$-bundles and the representation space. Here, to see that $q^{\ast}\bar{\mathcal{M}}$ is the moduli of $G$-bundles parametrized by curves over $q^{-1}(U)$ we use that $q^{-1}(U) \longrightarrow U$ is a covering space.

\par

The integrable almost complex structure on $ \bar {\mathcal{M}}$ gives an integrable almost complex structure on $q^{\ast}\bar{\mathcal{M}}$, which eventually gives an integrable almost complex structure on $q^{-1}(U) \times R$.  Now using that $q^{-1}(U)$ is dense in $\mathcal{T}_g$, if we have an integrable almost complex structure on $q^{-1}(U) \times R$ which extends to an almost complex structure on $\mathcal{T}_g \times R$, we can say the extended almost complex structure is integrable.

\par

Now what remains to show is that, the almost complex structure on $\mathcal{T}_g \times R$ we started with, coincides with this integrable almost complex structure we get from the moduli constructions. We shall denote $q^{-1}(U) \times R$ with the integrable almost complex structure as $\mathcal{M} '$.

\begin{lemma}
Fix a representation $\rho \in R$. Then we have a holomorphic section $s:q^{-1}(U) \longrightarrow \mathcal{M} '$, such that, $s(q^{-1}(U))=q^{-1}(U) \times \rho$ .
\end{lemma}

\begin{proof}
Since $\mathcal{T}_g$ is simply connected, the fundamental group of $\mathcal{C}_g$, curves parametrized by $\mathcal{T}_g$, is same as $\pi_1(S)$. Using the representation $\rho$, we can construct (construction is similar as for a single curve $S$) a holomorphic principal $G$-bundle parametrized by $\mathcal{T}_g$. This gives a holomorphic principal $G$ bundle parametrized by $q^{-1}(U)$. Hence, by the construction of moduli of $G$-bundles, we get a holomorphic map $s:q^{-1}(U) \longrightarrow \mathcal{M}'$ such that the image gives the $G$-bundles coming from the fixed representation $\rho$.
\end{proof}

\par

Using the above lemma we conclude that $i:q^{-1}(U) \times \rho \longrightarrow \mathcal{M}'$ is a holomorphic map $\forall \rho \in R$. Also for $t \in \mathcal{T}_g$, the almost complex structures of the fibres ${\mathcal{M}}_t'$ and ${\mathcal{M}}_t$ come from the complex structure of $M_{t}$ (moduli of $G$-bundles on the Riemann surface $X_t$). Hence we conclude that the almost complex structures of $\mathcal{M}'$ and $\mathcal{M}$ coincide, which finally gives the integrability of the almost complex structure on $\mathcal{M}$, as $q^{-1}(U)$ is dense in $\mathcal{T}_g$.

\par

 The holomorphicity of the map $\phi$ follows because it is a projection onto the first n coordinates of $\mathcal{M}$.

\section{Construction of universal adjoint bundle and determinant bundle}

\par

After constructing $\mathcal{M}$,  the universal moduli of  $G-$bundles over $\mathcal{T}_g$, in the previous section, our aim in this section is to construct the universal projective bundle over $\mathcal{C}_g \times _{\mathcal{T}_g}{\mathcal{M}}$ with a complex structure. We have, as introduced earlier, $G$ a semisimple linear algebraic group and $K \subset G$ to be a maximal compact subgroup. Let $Z$ denote the centre of $G$ and $K$ respectively. The centre of both $G$ and $K$ are same because $K$ is a maximal compact subgroup of $G$.

\par

We also fix the following notations: $X$ will denote a Riemann surface of genus $g \geq3$ with $S$ the underlying smooth surface. $M$ will denote the moduli of regularly stable principal $G-$bundles on $X$, $R$=$Hom_{ir}(\pi _1(X),K)/K$  and $\tilde{R}=Hom_{ir}(\pi _1 (X), K) $. 

\par

First we construct a smooth $K/Z$ bundle on $X \times M$ such that by extending the structure group we have a holomorphic $G/Z$ bundle. Then we have a unique connection $\nabla(t)$ on the $K/Z$ bundle corresponding to the holomorphic structure of the $G/Z$ bundle. Here the letter $t$ is used in $\nabla(t)$ to denote that $X$ and $M$ have the complex structure corresponding to the point $t \in \mathcal{T}_g$, the Teichm\"{u}ller space. 

\begin{lemma}
There exist a smooth principal $K/Z$ bundle $P$ over $X \times R$.
\end{lemma}

\begin{proof}
Let $\tilde{X}$ be the universal cover of $X$. Then $\tilde{X} \times \tilde{R} \times K$ is a trivial principal bundle over 
$\tilde{X} \times \tilde{R}$. We have an action of $\pi_1(X)$ on $\tilde{X} \times \tilde{R}$ given by a trivial action on 
$\tilde{R}$ and usual action on $\tilde{X}$. This action of $\pi_1(X)$ can be lifted to give an action of $\pi_1(X)$ on $\tilde{X} \times \tilde{R} \times K$. To see this we choose $\sigma 
\in  \pi_1(X)$, $f \in \tilde{R}$ and $k \in K$. Then $\sigma$ takes $(f,k)$ to $(f,f(\sigma)k)$. Let the map $\tilde p: \tilde{X} \times \tilde{R} \times K \longrightarrow \tilde{X} \times \tilde{R}$  be the projection map . Going modulo the action of $\pi_1(X)$ on $\tilde{X} \times \tilde{R} \times K$ and $\tilde{X} \times \tilde{R}$ we have a map $p:\tilde{P} 
\longrightarrow X \times \tilde{R}$, where $\tilde{P}$ is a principal $K$ bundle over $X \times \tilde{R}$ and the map $p$ is induced from $\tilde{p}$. 

\par

$\tilde{P}$  is a smooth K-bundle because the action of $\pi_1(X)$ is properly discontinuous. Recall that the representation space $R$ is the quotient of $\tilde{R}$ by $K$, where the action of $K$ on $\tilde{R}$ is via the conjugation action of $K$ on itself. $\bold {\underline {NOTE}}$ that the center $Z$, of $K$, is precisely the subgroup acting trivially because we are looking at irreducible representations of $\pi_1(X)$. 

\par

Recall that the action of $K$ on $\tilde{X} \times \tilde{R} \times K$ is via translation on $K$, hence the action does not have a fixed point. So the induced action of $K$ on $\tilde{P}$ also does not have any fixed point. We have already shown that $Z$ acts trivially on $X \times \tilde{R}$. Hence taking quotient of $\tilde{P}$ and $X \times \tilde{R}$ by $K$ with respect to the respective actions we have a map $\bar{p}: P \longrightarrow X \times R$ where $P(= \tilde{P}/K)$ is a  $K/Z$ bundle. 
\end{proof}

\par

In the following lemma we discuss a natural holomorphic structure on $P(G/Z)$, the extension of the structure group of $P$ by $G/Z$, which is induced from the holomorphic structures of $X,M$ and $G/Z$. Before we discuss the proof of the natural holomorphic structure on $P(G/Z)$, we state a lemma which would be needed in the construction of such holomorphic structure.

\begin{lemma}
Let $G,H$ be two linear algebraic groups. Let $X$ be a holomorphic manifold. Let $P_G$ denote a principal $G$-bundle on $X$ and $P_H$ denote a principal $H$-bundle on $P_G$ such that the action of $G$ on $P_H$ commutes with the action of $H$ on $P_H$. Let $D_G: T_{P_G} \longrightarrow T_{P_H}/H$ denote a $G$-invariant connection and $h: T_X \longrightarrow T_{P_G}/G$ be another connection. Then $D_G$ induces a connection $D: T_X \longrightarrow T_{P_H /G}/H$.   
\end{lemma}

\begin{proof}
We have the following diagram
\[
\begin{tikzcd}
P_H \arrow{d}{q} \arrow{r} & P_G \arrow{d} \\
P_H/G \arrow{r} & X
\end{tikzcd}
\]
where $q$ is the quotient map. Let $dq:T_{P_H} \longrightarrow T_{P_H/G}$ denote the corresponding map at the level of tangent spaces. Since the action of $G$ and $H$ commute, $dq$ is $H$ equivariant, and hence we have a map $\bar{dq}:T_{P_H}/G/H \longrightarrow T_{P_H/G}/H$. Since, by assumption, $D_G$ is $G$-invariant, we have $\bar{D_G}:T_{P_G}/G \longrightarrow T_{P_H}/H/G$. Then $D$ can be defined by $D=\bar{dq} \circ \bar{D_G} \circ h$. Thus giving the desired connection.
\end{proof}

\textbf{Note.} Sometimes we can construct the connection $D_G$, as mentioned in the above lemma, as a $G$-invariant connection. Even in that case the existence of section like $h$ is needed to construct connections on spaces we get through taking quotient by $G$ .   

\par

We recall that giving a complex structure on $P(G/Z)$ is equivalent to giving a connenction on $P(K/Z)$ with a $(1,1)$ curvature form. 

\begin{lemma}
$P(G/Z)$ exists as a holomorphic $G/Z$ bundle over $X \times M$. 
\end{lemma}

\begin{proof}
Let $E$ be the trivial principal $G$-bundle over the Riemann surface $X$. Let $Conn$ be the space of all \textit{irreducible flat} connections on $E$ giving $E$ the structure of a regularly stable principal bundle. $Conn$ is an infinite dimensional analytic space with a vector space structure. Let $Aut(E)$ denote the $C^{\infty}$ automorphisms of $E$. We know that $M=Conn/Aut(E)$ (see \cite{AB}) as holomorphic manifolds where $M$ denotes the moduli of regularly stable $G$-bundles.

\par

Consider the following diagram

\[
\begin{tikzcd}
X \times Conn \times K \arrow{r} \arrow{d}{q} & X \times Conn \arrow{d} \\
P(K/Z) \arrow{r} & X \times M
\end{tikzcd}
\]

\par

where the vertical maps are obtained by taking quotients by $Aut(E)$ action. 

\par

Note that for $k \in Z$, multiplication by $k$ induces an automorphism of $E$. Since $E$ is regularly stable for connections belonging to $Conn$, the subgroup of $G$ inducing automorphisms of $E$ is precisely $Z$. Also $X \times Conn$ is fixed by elements of $Z$. Hence taking a quotient of $X \times Conn \times K$ by $Aut(E)$ we have a principal $K/Z$ bundle over $X \times M$.

\par

There is the canonical connection $D:T_{X \times Conn} \longrightarrow T_{X \times Conn \times K}/K$ which can be described as follows. For a fixed point $t \in Conn$ $D$ is defined by $t$ and $D$ is defined as $identity$ for a fixed $x \in X$. Then $D$, by definition, becomes $Aut(E)$-equivariant.

\par

Also recall that there is a inclusion $T_{M} \hookrightarrow T_{Conn}/Aut(E)$ where $T_{M}$ is the space of all \textit{Harmonic forms}. Hence we get a connection $h: T_{X \times M} \longrightarrow T_{X \times Conn}/Aut(E)$. Then using Lemma 4.2. we have a connection on the $K/Z$-bundle  $P(K/Z)$. The curvature of this connection is a $(1,1)$-form.  This is because we can extend the structure group of the principal bundle using the adjoint map $G \longrightarrow Aut(\mathfrak{g})$ and the curvature for the corresponding connection is known to be a $(1,1)$-form, as in the case of vector bundles.  
 
\end{proof}

\par

The complex structure constructed above is the unique complex structure on the projective bundle $P(G/Z)$ over $X \times M$, satisfying the following conditions: \\
For any complex manifold $S$ and a map $f:S \longrightarrow P(G/Z)$ we have\\
1.
\[
\begin{tikzcd}
S \arrow{r}{f} \arrow[swap]{dr}{\tilde{f}} & {P(G/Z)} \arrow{d}{q} \\
              & {X \times M}
\end{tikzcd}
\]
$f$ has to be holomorphic, whenever $\tilde{f}$ is holomorphic.\\
2. The map $f_S: S \longrightarrow P'(G/Z) $, induced from $f$ via the pull-back diagram 
\[
\begin{tikzcd}
P'(G/Z) \arrow{r} \arrow[swap]{d} & P(G/Z) \arrow{d} \\
S \times{X} \times{M} \arrow{r}& X \times {M}
\end{tikzcd}
\] 
and
\[
\begin{tikzcd}
S \times{X} \times{M} \arrow{r} \arrow[swap]{d} & X \times {M} \arrow{d} \\
S \arrow{r} & pt
\end{tikzcd}
\] 
has to be holomorphic.

\par

We also have that $X \times \rho \hookrightarrow X \times M$ sits as a complex submanifold; as the inclusion, which is apriori just a smooth map, preserves the almost complex structures. The unique flat connection corresponding to the holomorphic structure of $P(G/Z)$ restricted to $X \times \rho$ is the same as $\nabla(t)$ restricted to $X \times \rho$. 

\par

So far we have fixed a holomorphic structure on the Riemann surface and have constructed a holomorphic principal $G/Z$ bundle $P(G/Z)$ over $X \times M$. Now our aim is to construct a holomorphic principal $G/Z$ bundle $\mathcal{P}$ on $\mathcal{C}_g \times _{\mathcal{T}_g}{\mathcal{M}}$ and study the properties of the curvature of its    holomorphic connection.

\par

Let $S$ be the underlying $C^ \infty$ surface of the Riemann surface $X$. Using the previous arguments we have a $C^ \infty$ projective bundle $P^{\infty}$ (the underlying $C^{\infty}$ surface of $P$ is $P^{\infty}$) on $S \times R$ with a partial flat connection along the tangent space of $S$. The construction of this partial flat connection is as follows: Fix a representation class $\rho \in R$. We can have a connection $h$ for the projection map $pr: T_{P^{\infty}} \longrightarrow T_S$ defined by $h(v)=(v,0)$ where $v \in T_{S}$. 

\par

The reason of this construction is that it would be independent of any choice of complex structure on $S$. We denote the projective bundle along with the partial flat connection as $(P^{\infty}, \nabla ({\mathcal{F}}))$. 

\par

Recall $Diff^{0} (S)$ to be the orientation preserving diffeomorphisms of $S$ which are homotopic to identity. $Diff^{0} (S)$ acts on $S \times R$ via the usual action on $S$ and the trivial action on $R$. In the next paragraph we shall describe an action of $Diff^{0} (S)$ on $P^{\infty}$ which is a lift of the above action on $S \times R$. 

\par

Choose $f \in Diff^{0} (S)$ and fix   $\rho \in R$. Let $(P^{\infty}(\rho), \nabla(\rho))$ be the projective bundle with the unique connection which comes as a restriction of $P$ on $S \times \rho$.  Consider the pullback bundle $f^{*}P^{\infty}(\rho)$ equipped with the pullback connection $f^{*} \nabla(\rho)$. (Note that the diffeomorphism $f$ of $S$ induces a diffeomorphism of $S \times \rho$.) Since $f \in Diff^{0} (S), (P^{\infty}(\rho), \nabla(\rho))$ and $(f^{*}P^{\infty}(\rho), f^{*}\nabla(\rho))$ are isomorphic. They are isomorphic upto a unique isomorphism because automorphisms of $(P^{\infty}(\rho), \nabla(\rho))$ is trivial (The automorphism group is trivial because the structure group has trivial centre). Hence $f$ induces an action on $P^{\infty}(\rho)$.

\par

The explicit description of the action of $f \in Diff^{0}(S)$ on $P^{\infty}(\rho)$ can be given as follows: Let $\phi_f : f^{\ast}P^{\infty}(\rho) \longrightarrow P^{\infty}(\rho)$ be the unique isomorphism. Then the action of $f$ on $P^{\infty}(\rho)$ takes a point $(s,\rho, k)$ to $\phi_f(f^{-1}(s), \rho, k)$. Clearly such action of $Diff^{0}(S)$ on $P^{\infty}(\rho)$ is a group action because $\phi_{f \circ g}=\phi_f \circ \phi_g$ for $f,g \in Diff^{0}(S)$. 

\par

For $f \in Diff^{0} (S)$ we have a unique isomorphism $I(f):(f \times Id)^{*}P^{\infty} \longrightarrow P^{\infty}$ such that over $S \times \rho$ it is the same isomorphism $\phi_f$ mentioned in the previous paragraph. We conclude here similarly as the previous paragraph that we have a goup action i.e. if $g \in Diff^{0} (S)$ we have the equality $I(f) \circ I(g)= I(f \circ g)$. Hence we have a lift of the action of $Diff^{0} (S)$ on $S \times R$ to the pair $(P^{\infty}, \nabla(\mathcal{F}))$. 
 
\par

Let $Com(S)$ denote the space of all complex structures on $S$. Let $p_{12}$ denote the projection of $S \times R \times Com(S)$ onto $S \times R$. Consider the projective bundle with a partial connection $(p_{12}^{*}P^{\infty}, p_{12}^{*} \nabla(\mathcal{F})) \longrightarrow S \times R \times Com(S)$. The group $Diff^{0} (S)$ acts on $S\times R \times Com(S)$. $Diff^{0} (S)$ acts on $S \times R$ by the previous action and acts on $Com(S)$ by the push-forward of a complex structure on $S$ by a diffeomorphism. Also the acton of $Diff^{0} (S)$ on $(P^{\infty}, \nabla(\mathcal{F}) )$ induces action of $Diff^{0} (S)$ on $(p_{12}^{*}P^{\infty}, p_{12}^{*} \nabla(\mathcal{F}))$.

\par

Consider the projection \\

\centerline{$p_{12}^{\ast}(P^{\infty})/Diff^0(S) \longrightarrow (S \times R \times Com(S))/Diff^0(S)$} 

We know that $(S \times R \times Com(S))/Diff^{0}(S) =\mathcal{C}_{g} \times R = \mathcal{C}_{g}\times _{\mathcal{T}_g} \mathcal{M}$. Hence we get a $K/Z$-bundle with a partial connection over $\mathcal{C}_{g}\times _{\mathcal{T}_g} \mathcal{M}$.
Let $(\mathcal{P}, \nabla(par))$ denote the projective $K/Z$ bundle with the partial connection $\nabla(par)$ over $\mathcal{C}_{g}\times _{\mathcal{T}_g} \mathcal{M}$ constructed above.

\par

Let $\mathcal{P}(G/Z)$ be the $G/Z$ bundle obtained from $\mathcal{P}$ by extending the structure group. In the next lemma we discuss the existence of complex structure on $\mathcal{P}(G/Z)$.

\begin{lemma}
$\mathcal{P}(G/Z)$ has a complex structure.
\end{lemma}

\begin{proof}
Our aim is to construct a connection on $\mathcal{P}$ with a $(1,1)$ curvature form which will give a corresponding unique complex structure on $\mathcal{P}(G/Z)$.The proof would be similar as in Lemma 4.3.

\par

Let $E$ be the trivial $C^{\infty}$ principal $K$-bundle over the Riemann surface $S$. Let $Conn$ be the space of all \textit{irreducible flat} connections on $E$ such that $E$ becomes regularly stable. Let $Aut(E)$ denote the $C^{\infty}$ automorphisms of $E$. We know that $M=Conn/Aut(E)$ (see \cite{AB}) as $C^{\infty}$ manifolds where $M$ denotes the moduli of regularly stable $G$-bundles.

\par

Consider the following diagram

\[
\begin{tikzcd}
\mathcal{C}_g \times Conn \times K \arrow{r} \arrow{d}{q} & \mathcal{C}_g \times Conn \arrow{d} \\
\mathcal{P} \arrow{r} & \mathcal{C}_g \times_{\mathcal{T}_g} \mathcal{M}
\end{tikzcd}
\]

\par

where the vertical maps are obtained by taking quotients by $Aut(E)$ action. 

\par

Note that for $k \in Z$, multiplication by $k$ induces an automorphism of $E$. Since $E$ is regularly stable the subgroup of $G$ inducing automorphisms of $E$ is precisely $Z$. Also $\mathcal{C}_g \times Conn$ is fixed by elements of $Z$. Hence taking a quotient of $\mathcal{C}_g \times Conn \times K$ by $Aut(E)$ we have a principal $K/Z$ bundle over $\mathcal{C}_g \times_{\mathcal{T}_g} \mathcal{M}$.

\par

We have the canonical connection $D:T_{\mathcal{C}_g \times Conn} \longrightarrow T_{\mathcal{C}_g \times Conn \times K}/K$ which is defined similarly as in Lemma 4.3.

\par

 Using Lemma 4.3, we see that for a fixed $t \in \mathcal{T}_g$ we have inclusion $T_{X_t \times M_{t}} \hookrightarrow T_{X_t \times Conn}/Aut(E)$. Hence we get a connection $h:T_{\mathcal{C}_g \times_{\mathcal{T}_g} \mathcal{M}} \longrightarrow T_{\mathcal{C}_g \times Conn}/Aut(E)$. Then using Lemma 4.2. we have a connection on the $K/Z$-bundle  $\mathcal{P}$. The curvature of this connection is a $(1,1)$-form because of the same reasons as in Lemma 4.3.  
\end{proof}

\par

Similar to the case of $X \times M$ as before, we know that $\mathcal{P}(G/Z)$ has a unique complex structure satisfying the following conditions:\\

For any complex manifold $S$ and a map $f:S \longrightarrow \mathcal{P}(G/Z)$ we have :\\
1.
\[
\begin{tikzcd}
S \arrow{r}{f} \arrow[swap]{dr}{\tilde{f}} & {\mathcal{P}(G/Z)} \arrow{d}{q} \\
              & {\mathcal{C}_g \times_{\mathcal{T}_g} \mathcal{M}}
\end{tikzcd}
\]
$f$ has to be holomorphic whenever $\tilde{f}$ is holomorphic.\\
2. The map $f_S: S \longrightarrow \mathcal{P}(G/Z) \times_{\mathcal{T}_g} S$, induced from $f$ via the pull-back diagram 
\[
\begin{tikzcd}
\mathcal{P}(G/Z) \times_{\mathcal{T}_g} S \arrow{r} \arrow[swap]{d} & \mathcal{P}(G/Z) \arrow{d} \\
S \arrow{r}& \mathcal{T}_g
\end{tikzcd}
\]
has to be a holomorphic map.

\par

Summing up the constructions , so far, we have obtained a holomorphic projective $G/Z$-bundle with the holomorphic connection $(\mathcal{P}(G/Z),\nabla)$.

\par

Let $ad(\mathcal{P})$ be the corresponding adjoint vector bundle ,of the $G/Z$ bundle $\mathcal{P}(G/Z)$, over $\mathcal{C}_g \times _{\mathcal{T}_g} \mathcal{M}$. Let $p_2: \mathcal{C}_g \times _{\mathcal{T}_g} \mathcal{M} \longrightarrow \mathcal{M} $ be the projection map. Let $\Theta$ be the determinant bundle on $\mathcal{M}$ of $ad(\mathcal{P})$ with respect to the projection map, defined as \\
\centerline{$\Theta:=\bigwedge ^{top}(R^1 p_{2\ast} ad(\mathcal{P}))$.} 

\par

The existence of the above defined determinant bundle, as a holomorphic line bundle, follows from the following general result which has been studied by Bismut, Gillet and Soul{\'e}:

\begin{theorem}
(Theorem 0.1 \cite{B1})Let $\pi:Y \longrightarrow B$ be a proper holomorphic map of complex analytic manifolds and let $\zeta$ be a complex holomorphic vector bundle on $Y$. Then the image of $\zeta$ by $\pi$ has a determinant say $\lambda$, which is a holomorphic line bundle $det R \pi _{\ast} \zeta$ on $B$.
\end{theorem}

\section{Curvature of the determinant line bundle}

We have constructed the holomorphic determinant bundle $\Theta$ in the previous section. Due to its holomorphic structure $\Theta$ has a holomorphic connection. We denote by $K(\Theta)$ the curvature of the holomorphic connection. 

\par

We know that the holomorphic tangent space of  $M$ at a point $\rho \in M$ can be identified as the space of deformations of the $G$-bundle $\rho$ and hence can be identified with $H^1(X, ad(\rho))$, where $ad(\rho)$ is the adjoint bundle of the principal bundle corresponding to the point $\rho$. $\Omega$ is the 2-form corresponding to the natural metric on $M$ which we get by identifying the tangent space of $M$ with $H^{1}(X,ad(\rho))$ and using the cup product for $H^1(X, ad(\rho))$.

\par

In this section we prove the following statement:

\begin{theorem}
$K(\Theta)=-2 \pi \iota . q^{\ast} \Omega +(r/{6 \pi} ) \iota . \sigma ^{\ast} \omega$
\end{theorem}

Let $\Phi$ denote the R.H.S. of the above equation. The notations involved in $\Phi$ are as follows: $r=rank(ad(\mathcal{P}))$ (which is the same as $dim(Lie(G))$). $\iota = \sqrt{-1}$.  $\omega$ is a K{\"a}hler $(1,1)$ form on $\mathcal{T}_g$ corresponding to the $Weil-Petersson$ metric. The maps $q$ and $\sigma$ are obvious projections of $\mathcal{M}$ onto $R$ and $\mathcal{T}_g$ respectively.

\par

Before beginning the proof of the theorem we wish to recall a few results involving $chern$ $class$ and $Todd$ $class$ which would be used in the proof. See \cite{Ko} for details. 

\par

Let $Y$ be a complex manifold and $E$ be a holomorphic vector bundle on $Y$. We shall use the following notations: $Ch^i$ will denote the $i-th$ chern form which is a differential 2$i$-form on $Y$, $ch^i$ will denote the $i-th$ chern class which is the class of $Ch^i$ in the $deRham$ cohomology $H^{2i}(Y, \mathbb{Z})$. Moreover $Ch$ will denote the chern character for the vector bundle $E$. Similarly $Td^i$ will denote the $i-th$ Todd form and $td^i$ will denote the $i-th$ Todd class. We note down a couple of lemmas (proved in Kobayashi \cite{Ko}) involving chern forms and chern character.

\begin{lemma}
Let $E$, $F$ be two holomorphic vector bundles on $Y$. Then $Ch(E \otimes F)= Ch(E).Ch(F).$ 
\end{lemma}

\begin{lemma}
For a holomorphic vector bundle $E$ of rank $r$, $Ch(E)=r+Ch^1(E)+\frac{1}{2}(Ch^1(E)^2 -2.Ch^2(E))+...$ .
\end{lemma}

\par

Next we state the result (Theorem 0.1 \cite{B1}), due to Bismut-Gillet-Soul{\'e}, which we would be using to calculate $K(H_Q)$ .

\begin{theorem}
Let $\pi:Y \longrightarrow B$ be a proper holomorphic map of complex analytic manifolds and let $\zeta$ be a complex holomorphic vector bundle on $Y$. Let $\lambda$ be the determinant of $\zeta$ by $\pi$. For $b \in B$ let $Z$ denote the fibre. Let $R^Z$, $L^{\zeta}$ be the curvatures of the holomorphic hermitian connections on $T^{(1,0)}Z$ and $\zeta$. Then the curvature of the holomorphic connection on $\lambda$ is the component of degree 2 in the folowing form on $B$:\\
\centerline{$2 \pi \iota \int_{Z} Td(-R^Z/2 \pi \iota) Tr[exp(-L^{\zeta}/2 \pi \iota)].$}
\end{theorem}

\par

First we give a brief outline of a proof of a special case of the theorem, due to Quillen \cite{Q}. Let $X$ be a Riemann surface of genus $g$. $E$ be a smooth vector bundle of rank $r$ and degree $d$ over $X$. Let $\mathcal{B}=\Omega^{0,1}(\mathcal{E}nd(E))$ and $\mathcal{A}$ be the space of $\bar{\partial}$ operators which is an affine space relative to $\mathcal{B}$. Let $B \in \mathcal{B}$ be of the form $\alpha(z)d\bar z$. Let $\alpha(z)^{\ast}$ be the adjoint of $\alpha(z)$ relative to the local orthonormal frame of $E$. Using the adjoint $\alpha(z)^{\ast}$ of $\alpha(z)$ we construct an element of $\Omega^{1,0}(\mathcal{E}nd(E)$, denoted by $B^{\ast}:=\alpha(z)^{\ast}dz$. Then $tr(B^{\ast}.B)$ is a $(1,1)$ form on $X$ and hence we get a metric on $\mathcal{A}$ defined by \\
\centerline{$\vert B \vert ^2= \int _X \frac{\iota}{2 \pi}tr(B^{\ast}.B)$.}

\par

Fixing the notations as above we give a brief outline of a proof for the following result:

\begin{theorem}
\textbf{(Quillen, \cite{Q})} The curvature of the determinant line bundle (constructed by Quillen) is equal to the K{\"a}hler form obtained from the above mentioned metric on $\mathcal{A}$.   
\end{theorem}

\begin{proof}
We know that for a holomorphic line bundle the curvature form of the holomorphic connection is given by $\bar{\partial} \partial log \vert \sigma \vert^2$ where $\sigma$ is any local holomorphic section. Before calculating the curvature form we recall a few identities (see \cite{Q} for details). In terms of local orthonormal frame of $E$ we have the following local equations:

\begin{itemize}
\item{$ds^2=\rho(z) \vert dz \vert^2$;}\
\item{$D=d \bar{z}(\partial_{\bar{z}}+ \alpha)$, where D is a $\bar{\partial}$ operator;}\
\item{$\nabla=dz(\partial _z - \alpha^{\ast})+d\bar{z}(\partial_{\bar{z}}+ \alpha)$;}\
\item{$G(z,z')=\frac{\iota dz'}{2\pi (z-z')}[1+(z-z')\beta(z')-{(\bar{z}-\bar{z}')} \beta(z')+...];$}\
\end{itemize}

\par

Let $J=\frac{\iota}{2\pi} dz(\beta - \alpha^{\ast} - \frac{1}{2}\partial_{z}log(\rho))$. In \cite{Q} Quillen gives a different definition for $J$ and then uses the above local equations to get this expression. We would need only this expression for the calculation of the curvature form. We fix a one-parameter family $D_w$ of invertible $\bar{\partial}$ operators depending holomorphically on the complex variable $w$. Since invertible operators are dense in $\mathcal{E}nd(E)$, enough to show that the curvature form and the K{\"a}hler form coincide over such a family.

\par

We have the curvature form $\bar{\partial}\partial log \vert \sigma \vert^2=d \bar{w} dw\frac{\partial^2}{\partial_{\bar{w}} \partial_w} \zeta'(0)$,  where $\zeta(s)=tr(\Delta^{-s})$ and $\Delta=D^{\ast}.D$. Now,\\
$-\partial_w \zeta(s)=s.tr(\Delta^{-s-1} \partial_w \Delta)\\
=s.tr(\Delta^{-s}D^{-1} \partial_w D)$ ($\partial_w D^{\ast}=0$, as $D$ depends holomorphically on $w$)\\ 
$=\frac{s}{\Gamma(s)} \int_{0}^{\infty}tr(exp(-t\Delta) D^{-1}\partial_w D)t^{s-1}dt \\
=s.[\int_M tr(J \partial_w D) +O(s)] $ as $s \to 0$. \\
The last equality follows from Theorem 2 of \cite{Q}, which states that for $B \in \mathcal{B}$ we have the equality \\
\centerline{$\lim_{t \to 0}tr(exp(-t\Delta)D^{-1}B)=\int_M tr(JB)$.} 

\par

Hence it follows that $-\partial_w \zeta ' (0)=\int tr(J \partial_w D).$ We also have the equalities:\\

\centerline{$\partial_{\bar{w}}J=-\frac{\iota}{2\pi}dz \partial_{\bar{w}} \alpha^{\ast}=-\frac{\iota}{2\pi}(\partial_w D)^+.$}

The first equality holds because only $\alpha^{\ast}$ is not holomorphic in $w$. The second equality holds because $\partial_w D=\partial_w \alpha d\bar{z}.$ 

\par

Hence finally we have:\\

\centerline{$\frac{\partial^2}{\partial_{\bar{w}} \partial_w} \zeta'(0)=-\int \partial_{\bar{w}}tr(J \partial_w D)=\int \frac{\iota}{2\pi} tr[(\partial_w D)^+ \partial_w D] $.}

Identifying the vector space $\mathcal{B}$ with its tangent space we get the desired result.
\end{proof}

\par

Next we deduce the following result from the theorem which we would be using later. For a Riemann surface $X$, let $M_{V}$ be the moduli of stable vector bundle of $rank=r$ and $degree=d$ on $X$. Since a vector bundle can be seen as a principal $GL_n$-bundle, we get a universal projective bundle $P_V$ over $X \times M_V$. Let $ad(P_V)$ be the adjoint bundle of the $PGL_n$ bundle $P_V$.  

\begin{corollary}
The curvature $K(ad(P_V))$ of the holomorphic determinant bundle with respect to the projection $p:X \times M_{V} \longrightarrow M_{V}$ is equal to $4 \pi \iota r \Omega _V$, where $\Omega _V$ is a $(1,1)$ form on $M_{V}$ defined as in the beginning of the section.   
\end{corollary}

\begin{proof}
For any vector bundle $E$ over $X$, let $\mathcal{A}_s$ be the collection of stable holomorphic structures on $E$. Then $\mathcal{A}_s$ would be an open subset of $\mathcal{A}$, which is the collection of all complex structures on $E$. There is a natural action of the complex gauge group $\mathcal{PG}=Aut(E)/\mathbb{C}^{\ast}$ on $\mathcal{A}$, which, when restricted to $\mathcal{A}_s$ is a free action. Here $Aut(E)$ is the group of smooth automorphisms of $E$. Note that $Aut_{hol}(E)$, the holomorphic automorphisms of $E$ with a fixed stable structure is precisely $\mathbb{C}^{\ast}$. The quotient space of $\mathcal{A}_s$ by the action $\mathcal{PG}$ is $M_{V}$. 

\par

The Quillen determinant line bundle on $\mathcal{A}_s$ descends down to a holomorphic line bundle $L$ on $M_{V}$. This is true because the action of $Aut(E)/\mathbb{C}^{\ast}$ on $\mathcal{A}_s$ lifts to give an action on the Quillen determinant line bundle. The line bundle $L$ inherits a natural Quillen metric from the Quillen metric of the determinant bundle over $\mathcal{A}_s$. Using the results of Bismut-Gillet-Soul{\'e} we know that the construction of the determinant bundle by Quillen and the determinant bundle of the direct image of $ad(P_V)$ are isomorphic as holomorphic line bundles. Using Theorem 5.5 we see that the curvature of the holomorphic line bundle equals to a particular K\"{a}hler form on $M_V$. One can calculate (Proposition 2.9, \cite{B}) that the metric induced on $M_V$ from the Quillen metric on $\mathcal{A}_s$ is the same as $4 \pi \iota r \Omega_V$ . Hence we have $K(ad(P_V))=4 \pi \iota r \Omega_V$.
\end{proof}

\par

We note down a simple lemma which will used in the proof of Theorem 4.1 .

\begin{lemma}
Let  $X$ be a complex manifold. $G,H$ be linear algebraic groups and $f:G \longrightarrow H$ be a map of linear algebraic groups. Let $P$ be a principal $G-$bundle on $X$. Let $\mathfrak{g,h}$ be the lie algebras of $G,H$ respectively. Then the following diagram commutes:
\[
\begin{tikzcd}
H^i(X,P(\mathfrak{g})) \times H^i(X,P(\mathfrak{g})) \arrow{r}{\cup} \arrow[swap]{d}{f_{\ast}} & H^{2i}(X,P(\mathfrak{g})) \arrow{d}{f_{\ast}}\\
H^i(X,P(\mathfrak{h})) \times H^i(X,P(\mathfrak{h})) \arrow{r}{\cup} & H^{2i}(X,P(\mathfrak{h}))
\end{tikzcd} 
\]
\end{lemma}

\begin{proof}
Follows from the functoriality of the cup product.
\end{proof}

\begin{proof}
\textbf{(of Theorem 4.1)} The result of Bismut-Gille-Soule applied to our situation says $K(\Theta)=2 \pi \iota . p_{2_{\ast}}( Td^2.Ch^0 +Td^1.Ch^1+Td^0.Ch^2)$, where $p_{2_{\ast}}$ denotes integrating along the fibres which are curves. Note that since we are integrating along a curve , which has real dimension=2, the degree 2 component of $2 \pi \iota \int_{Z} Td(-R^Z/2 \pi \iota) Tr[exp(-L^{\zeta}/2 \pi \iota)]$ is precisely $2 \pi \iota . p_{2_{\ast}}( Td^2.Ch^0 +Td^1.Ch^1+Td^0.Ch^2)$.

\par

We recall that $Ch^i$ are polynomials in the curvature form corresponding to the holomorphic connection on $ad(\mathcal{P})$. Also recall that $Td^i$ are polynomials in the curvature of the hermitian connection on $T^{(1,0)}Z$. 

\par

Since $ad(\mathcal{P})$ is an adjoint bundle, $Ch^1=0$. Also we have $Td^0=1$ and $Ch^0=r$. We observe that, since $Td^i$ depends on the curvature of the hermitian connection on $T^{(1,0)}Z$, $Td^2$ depends only on the fibre and not on the holomorphic vector bundle $ad(\mathcal{P})$. 

\par

Now with the aim of calculating $Td^2$ and $Ch^2$ we note a few more properties of them. Let $\alpha =(t,\rho) \in \mathcal{T}_g \times R$. Let $v,w \in T_t \mathcal{T}_g, T_{\rho}R$ respectively. Let $v', w'$ be the image of $v,w$ respectively in $T_{\alpha}\mathcal{M}$. Recall that $T_{\alpha}\mathcal{M}=T_{t}\mathcal{T}_g \oplus T_{\rho}M_{t}$, where $M_{t}$ is the moduli of $G$-bundles on $X$ corresponding to the complex structure $t \in \mathcal{T}_g$. Then we have : 

\begin{lemma}
$p_{2_{\ast}}Td^2(w')=0$.
\end{lemma}

\begin{proof}
The 2-form $p_{2_{\ast}}Td^2$ doesn't depend on the coordinates of $R$. This is because the coordinates of the fibre are only dependent on the coordinates of $S$ and $\mathcal{T}_g$ and hence , in particular, the curvature of the hermitian connection on $T^{(1,0)}(Z)$ is independent of the coordinates of $R$. Hence we have that $p_{2_{\ast}}Td^2$  is a pullback of a 2-form on $\mathcal{T}_g$. This implies that $p_{2_{\ast}}Td^2(w')=0$.
\end{proof}

\begin{lemma} 
$p_{2_{\ast}}Ch^2(v')=0$.
\end{lemma}

\begin{proof} 
Let $i:\mathcal{C}_g \times \rho \longrightarrow \mathcal{C}_g \times_{\mathcal{T}_g} \mathcal{M}$ be the inclusion. Since the connection to the projective bundle, constructed earlier, is flat when restricted to $\mathcal{C}_g \times \rho$        we have \\

\centerline {$Ch^2(v')=0$} 

where $v'$ by abuse of notation will also denote the image of $v$ in the tangent space of $\mathcal{C}_g \times \rho$. Since $Ch^2(v')=0$ for all tangent vectors of $\mathcal{T}_g$, we conclude that $Ch^2$ is a pullback of a 4-form on $S \times R$. Hence $p_{2_{\ast}}Ch^2$ is a pullback of a 2-form on $R$. Hence we conclude that  $p_{2_{\ast}}Ch^2(v')=0$ for all $v' \in T_t{\mathcal{T}_g}$.
\end{proof}

\par

We conclude from these observations that $K(\Theta)(v',w')=0$ where $v'$ and $w'$ are as described above in the conditions of the above two lemmas. We also conclude that:\

\begin{itemize}
\item{$K(\Theta)$ and $2 \pi \iota p_{2_{\ast}}.rTd^2$ coincide on $\mathcal{T}_g \times \rho$.}\
\item{$ K(\Theta)$ and $2 \pi \iota p_{2_{\ast}}.Ch^2$ coincide on $t \times R$.}
\end{itemize}

\par 

To calculate the first case we observe that $Td^2$ does not depend on the adjoint bundle. Hence if we replace the adjoint bundle by the trivial bundle with the trivial metric, $Td^2$ remains the same. The only difference which occurs while caculating $K(\Theta)$is that $Ch^0=1$ for the trivial bundle and $Ch^0=r$ for the adjoint bundle. Hence \\

\centerline{$K(\Theta) \mid_{\mathcal {T}_g}=r.K(\Theta)_{triv}$.}

  $K(\Theta)_{triv}$ is the curvature of the determinant line bundle on $\mathcal{T}_g$ for the trivial bundle on $\mathcal{C}_g$. We know that $K(\Theta)_{triv}=(\iota /6\pi) \omega$ (see \cite{ZT}). Recall that $\omega$ is the K\"{a}hler $(1,1)$ form corresponding to the \textit{Weil-Petersson} metric on  $\mathcal{T}_g$. Hence $K(\Theta)= (r.\iota /6\pi). \sigma^{\ast} \omega$ when restricted to $\mathcal{T}_g \times \rho$.

\par

To get a complete expression of $K(\Theta)$, what remains is, to calculate $K(\Theta)$ when restricted to $t \times R$ ,where $t \in \mathcal{T}_g$. As defined earlier, let $X, M, M_{V}$ denote the Riemann surface (corresponding to the complex structure $t$), moduli of stable $G-$bundles on $X$ and the moduli of vector bundles of rank$=r$ on $X$ respectively. Let $f:M_G \longrightarrow M_V$ be the map which takes a principal bundle $\rho$ to the adjoint vector bundle $ad(\rho)$. 

\par

To calculate $K(\Theta)$ restricted to $t \times M$ we need to calculate $p_{2_{\ast}} Ch^{2}(ad(P))$. Recall that $P$ is the universal principal $G/Z$-bundle which had been constucted over $X \times M$. To calculate $p_{2_{\ast}} Ch^{2}(ad(P))$ we observe that:

\begin{itemize} 
\item{$Ch^2(ad(P) \otimes ad(P))=-2r.Ch^2(ad(P))$. Using lemma 5.2 and lemma 5.3.} 
\item{Since we are integrating the chern form along the fibre we can as well take the corresponding chern class. Hence $q_{2_{\ast}}Ch^2(.)$ would not depend on the choice of a metric, and hence not on the choice of a connection on the vector bundle as well.} 
\end{itemize}

\par

Next we consider the vector bundle $\mathcal{E}nd$, the endomorphism bundle over $X \times M_{g,V}$. We know that $f^*(\mathcal{E}nd)=ad(P)^2$  in the sense that the Einstein-Hermitian connection pulls back to the Einstein-Hermitian connection. Where we recall $f:M \longrightarrow M_{V}$ to be the adjoint map i.e. the map takes a principal bundle to its adjoint bundle. Let $Det(\mathcal{E}nd)$ be the determinant bundle on $M_{V}$ corresponding to the vector bundle $\mathcal{E}nd$ and $K(\mathcal{E}nd)$ be its curvature. 

\par

We have $f^{\ast}(Det(\mathcal{E}nd))=Det(ad(P)) \otimes Det(ad(P))$, since we know that $Det(.)$ behaves well with respect to pull back. More precisely, suppose we have a commutative diagram:

\[
\begin{tikzcd}
X_1 \arrow{r}{f_X} \arrow{d} & X_2 \arrow{d}\\
Y_1 \arrow{r}{f_Y} & Y_2
\end{tikzcd}
\]
and let $E$ be a vector bundle on $X_2$. Then $f_{Y}^{\ast}(Det(E))=Det(f_{X}^{\ast}(E))$. 

\par

Hence we have $f^{\ast}(K(\mathcal{E}nd)) = -2r.K(\Theta)|_{t \times R}$ as 2-forms over $M$.

\par

So we reduce the problem to calculating $K(\mathcal{E}nd)$. We observe that\\ 

\centerline{$\mathcal{E}nd=ad(P_V) \oplus L$,} 

where $L$ is the trivial line bundle $X \times \mathbb{C}$ and $P_V$ is the universal projective bundle. 

\par

On $\mathcal{E}nd$ we give the connection   $D$ induced from the Einstein-Hermitian connection $D_1$ on $P_V$ and the connection $D_2$ on $L$ defined by $D_2(f.c)=df.c$, where $f$ is a smooth function on $X$ and $c$ is a constant section. Since $D$ and $D_1$ have the same curvature form, $\mathcal{E}nd$ and $P_V$ have the same chern form $Ch^2$. The corresponding chern class does not depend on the choice of the connection. Hence using the formula of Bismut-Gillet-Soul{\'e}, for calculating the curvature of the determinant bundle, we conclude that $K(\mathcal{E}nd)=K(ad(P_V))=4 \pi r \iota . \Omega _V$ . 

\par

We have $f^*(\Omega _V)=\Omega$ which follows from lemma 5.7 by taking $H=Aut(\mathfrak{g})$, where the map from $G$ to $Aut(\mathfrak{g})$ is given by adjoint action. Hence we finally have \\
\centerline{$K(H_Q)=-2 \pi  \iota. q^*\Omega+ \frac{r}{6 \pi} \iota. \sigma^* \omega$.}
 
\end{proof}

\end{document}